\documentclass[11pt,letterpapert]{article}
\usepackage{graphicx}
\usepackage{amssymb}
\usepackage{epstopdf}
\usepackage{fancyhdr}
\usepackage{amsmath}
\usepackage{amsthm}
\usepackage{hyperref}
\usepackage{makeidx}
\usepackage{blkarray}
\usepackage{mathdesign}
\usepackage{color}
\usepackage{geometry}
\usepackage{soul}
\usepackage{nameref}
\usepackage{ulem}

\definecolor{purple}{rgb}{.9,0,.9}

\definecolor{newred}{RGB}{180,20,5}
\def\clr{\color{newred}}
\definecolor{newgreen}{RGB}{1,129,30}

\graphicspath{{./figures/}}

\makeatletter
\let\orgdescriptionlabel\descriptionlabel
\renewcommand*{\descriptionlabel}[1]{%
  \let\orglabel\label
  \let\label\@gobble
  \phantomsection
  \edef\@currentlabel{#1}%
  \let\label\orglabel
  \orgdescriptionlabel{#1}%
}
\makeatother

\graphicspath{{./figures/}}

\newcommand{\Nat}{\mathbb{N}}

\newcommand{\Real}{\mathbb{R}}

\newcommand{\rfa}{\quad {\rm for \ all}\ }

\newcommand{\cA}{{\cal A}}\newcommand{\cB}{{\cal B}}\newcommand{\cC}{{\cal C}}

\newcommand{\cI}{{\cal I}}

\newcommand{\cS}{{\cal S}}\newcommand{\cT}{{\cal T}}\newcommand{\cU}{{\cal U}}
\newcommand{\cX}{{\cal X}}

\newcommand{\ba}{{\bf a}}
\newcommand{\be}{{\bf e}}

\newcommand{\bn}{{\bf n}}

\newcommand{\bu}{{\bf u}}
\newcommand{\bv}{{\bf v}}\newcommand{\bw}{{\bf w}}\newcommand{\bx}{{\bf x}}

\newcommand{\bH}{{\bf H}}
\newcommand{\bL}{{\bf L}}

\newcommand{\tr}{\text{tr}}

\newcommand{\ve}{\varepsilon}

\newtheorem{theorem}{Theorem}[section]
\newtheorem{lemma}[theorem]{Lemma}

\newtheorem{proposition}[theorem]{Proposition}

\def\XXint#1#2#3{{\setbox0=\hbox{$#1{#2#3}{\int}$ }
\vcenter{\hbox{$#2#3$ }}\kern-.6\wd0}}

\newcommand{\beqn}{\begin{equation}}
\newcommand{\eeqn}{\end{equation}}

\newcommand{\bone}{\textbf{1}}

\title{On a  notion of nonlocal curvature tensor}

\author{
Roberto Paroni$^1$ \!\!\!\!\! \and Paolo Podio-Guidugli$^{2}$ \!\!\!\!\! \and Brian Seguin$^3$
}
\begin{document}

\maketitle

\vspace{-1cm}
\begin{center}
{\small
$^1$ DICI,
Universit\`a di Pisa\\ 
Largo Lucio Lazzarino 1, 56122 Pisa, Italy\\
\href{mailto:roberto.paroni@unipi.it}{roberto.paroni@unipi.it}\\[8pt]
$^2$ Accademia Nazionale dei Lincei\\ 
Palazzo Corsini, Via della Lungara 10, 00165 Roma, Italy\\
\href{mailto:p.podioguidugli@gmail.com}{p.podioguidugli@gmail.com}\\[8pt]
$^3$  Department of Mathematics and Statistics\\ 
Loyola University Chicago, Chicago, IL 60660, USA\\
\href{mailto:bseguin@luc.edu}{bseguin@luc.edu}\\[8pt]
}
\end{center}

\bigskip
\begin{center}
{\it Dedicated to Roger L.\ Fosdick\\ in recognition of his lifelong unvaluable service to our scientific community}
\end{center}
\medskip

\begin{abstract}
In the literature various notions of nonlocal curvature can be found. Here we propose a notion of nonlocal curvature tensor. This we do
by generalizing an appropriate representation of the classical curvature tensor and by exploiting some analogies with certain fractional differential operators.

\end{abstract}

\section{Introduction}
The objects of classical differential geometry are curves, surfaces, and curves on surfaces. Their study is {\it{local}}, in terms of certain first- and second-order differential characters, which for surfaces are the surface metric tensor and the curvature tensor, respectively. The latter tensor carries the information of interest for most applications, via its trace, the mean curvature, and its determinant, the Gaussian curvature. The classical notion of mean curvature offers a well known bridge between geometry and analysis, in that it  is deducible from the
stationarity condition of the area functional, a condition crucial to solve a central problem in modern calculus of variations, the minimal surface problem.

To date, various notions of {\it{nonlocal}} curvature have been proposed for surfaces. To our knowledge, the only proposition for curves has been made in \cite{S20} (see \cite{S20a} for an updated, corrected version of this work). The common fundamental idea of the numerous papers dealing with 
surfaces is to introduce a fractional notion of area and find what condition(s) must hold point-wise on the surface to ensure that the fractional area has a minimum subject to some boundary condition. In this paper we propose a notion of {\it{nonlocal curvature tensor}}. Our present work is a follow-up of 
\cite{PPGS}, a paper prompted by \cite{PPGL}, inspired by our reading of \cite{AV}, and written in the wake of a quite abundant literature, part of which we briefly review in the next section (see also \cite{Alb,AdPM,LC, LCMag, CS,CVb,Merri}).

In Section 2 we discuss the notions of nonlocal mean and directional curvatures at a typical point of a surface, both when the surface in question is the boundary of a bounded and compact set and when it is not. In Section 3 we motivate and introduce a new notion of {\it{nonlocal curvature tensor}}, which we illustrate by computing it for a $n$-dimensional sphere. In Section 4 we set forth some conjectures about representing such nonlocal curvature tensor in terms of fractional differential operators. The mathematical machinery from fractional calculus we employ is partly standard, quickly recapped in Section 4.1; the representation of a second fractional gradient of a scalar-valued function exposed in Section 4.2 is new. Unfortunately,  we conclude that our conjectures, no matter how carefully motivated, cannot be verified in their present form. We then finish with another conjecture, namely that the nonlocal curvature tensor should be defined in terms of a---at the moment of this writing---nonexisting notion of fractional surface gradient of the normal field. 
\section{Notions of nonlocal curvature for a surface}

In this section we summarize some information about the nonlocal geometry of surfaces, both when they bound a set and when they do not, and hence have a boundary. 

The perimeter measure of a bounded set $E$ with nice boundary $\partial E$ is the same as the area measure of $\partial E$. An identical, at bottom purely terminological, alternative occurs when fractional counterparts of these notions are introduced.
For each $\sigma$ between 0 and 1,  $\sigma$-$\text{Per}$, the {\it{fractional perimeter}}, also called $\sigma$-perimeter, of a measurable set $E\subseteq\Real^n$ relative to a bounded set $\Omega$, is delivered by the functional
\beqn\label{i}
\sigma\text{-Per}(E,\Omega):=\cI(E\cap\Omega,\cC E\cap\Omega)+\cI(E\cap\Omega,\cC E\cap\cC\Omega)+\cI(E\cap\cC\Omega,\cC E\cap\Omega),
\eeqn
where the integral
\beqn\label{ii}
\cI(A,B):=\frac{1}{\alpha_{n-1}}\int_A\int_B\frac{1}{|x-y|^{n+\sigma}} dxdy
\eeqn
can be interpreted as a geometric distance interaction between the sets $A$ and $B$.
Caffarelli and Valdinoci \cite{CVa} showed that if $\partial E\cap B_R$ is $C^{1,\beta}$ for $\beta\in(0,1)$ and $B_R$ a ball of radius $R$, then
\beqn\label{sPerlim}
\lim_{\sigma\rightarrow 1^-} (1-\sigma)\sigma\text{-Per}(E,B_r)=\text{Per}(E,B_r)
\eeqn
for almost every $r\in(0,R)$.  

A set $E$ is a minimizer of the $\sigma$-perimeter relative to $\Omega$ if
\beqn
\sigma\text{-Per}(E,\Omega)\leq \sigma\text{-Per}(F,\Omega)
\eeqn
for all measurable sets $F$ such that $E\setminus \Omega=F\setminus\Omega$, meaning that $E$ and $F$ agree outside of $\Omega$.  
It was shown by Caffarelli, Roquejoffre, and Savin \cite{CRS} that if $E$ is a minimizer, then it must satisfy 
\beqn \label{a}
\int_{\Real^n} \frac{\tilde\chi_E(y)}{|z-y|^{n+\sigma}} dy=0\rfa z\in\partial E,
\eeqn
where 
\beqn\label{chibar}
\tilde\chi_E:=\chi_E-\chi_{\cC E}
\eeqn
 is a difference of characteristic functions.\footnote{The integral in \eqref{a} must be understood in the principal value sense as the integrand has a singularity at $y=z$.}  Consideration of the relationship in \eqref{sPerlim} between  $\sigma$-perimeter and classical perimeter and of the well known fact that surfaces minimizing their perimeter have boundaries with zero mean curvature, motivates defining a {\it{fractional}}, or {\it{nonlocal}}, {\it{mean curvature}} at $z\in\partial E$ by
\beqn\label{Hsper}
H_\sigma(z):=\frac{1}{\omega_{n-2}}\int_{\Real^n}  \frac{\tilde\chi_E(y)}{|z-y|^{n+\sigma}} dy,
\eeqn
where $\omega_{n-2}$ is the $(n-2)$-dimensional measure of the unit sphere in $\Real^{n-1}$.  The factor used in this definition is suggested by a result of Abatangelo and Valdinoci's \cite{AV}: if $E$ has a smooth boundary, then
\beqn\label{limHs}
\lim_{\sigma\rightarrow 1^-} (1-\sigma)H_\sigma(z)=H(z),
\eeqn
where $H(z)$ is the classical mean curvature at $z\in\partial E$.

The above discussion involves $(n-1)$-dimensional surfaces that are the boundary of a set.  However, these ideas can be extended to orientable surfaces that are not the boundary of a set. 

Let $\cS$ be a compact $C^1$ surface with or without boundary in $\Real^n$, and let a unit-vector valued function $\bn$ defined on $\cS$ define its orientation. 
For $0<\sigma<1$ and $\Omega$ a bounded open set, the $\sigma$-Area of $\cS$ relative to $\Omega$ was defined in 
\cite{PPGS} to be
\beqn\label{sa}
\sigma\text{-Area}(\cS,\Omega):=\frac{1}{2\alpha_{n-1}}\int_{\cX(\cS)} \frac{\max\{\chi_\Omega(x),\chi_\Omega(y)\}}{|x-y|^{n+\sigma}}dxdy,
\eeqn
where $\alpha_{n-1}$ is the volume of the unit ball in $\Real^{n-1}$ and $\cX(\cS)$ is the set of all pairs $(x,y)\in\Real^n\times\Real^n$ such that the line segment joining $x$ and $y$ crosses $\cS$ an odd number of times.  

When $\cS=\partial E$, we have that
$$
\cX(\cS) = (E\times \cC E)\cup(\cC E\times E)
$$
up to a set of measure zero in dimension $2n$, see \cite{PPGS}. Moreover,  we deduce from \eqref{sa} that 
$$
\sigma\text{-Area}(\cS,\Omega)= 
\cI(E\cap\Omega,\cC E\cap\Omega)+\cI(E\cap\Omega,\cC E\cap\cC\Omega)+\cI(E\cap\cC\Omega,\cC E\cap\Omega)
$$
and hence from \eqref{i}, it follows that
$$
\sigma\text{-Area}(\cS,\Omega)=\sigma\text{-Per}(E,\Omega)
$$
and therefore the $\sigma$-Area generalizes the $\sigma$-Per.

The condition necessary and sufficient for the vanishing of the first variation of the $\sigma$-Area functional with respect to surfaces with the same boundary is a pointwise condition which was used in \cite{PPGS} to motivate the following notion of {\it{nonlocal mean curvature}}:
\beqn\label{Hs}
H_\sigma(z):=\frac{1}{\omega_{n-2}}\int_{\Real^n} \frac{\widehat\chi_\cS(z,y)}{|z-y|^{n+\sigma}}dy\rfa z\in\cS,
\eeqn
where 
\beqn\label{chiS}
\widehat\chi_\cS(z,y):=
\begin{cases}
1 & y\in\cA_i(z),\\
0 & y\not\in \cA_i(z)\cup\cA_e(z),\\
-1 & y\in\cA_e(z),
\end{cases}
\eeqn
and
\begin{align*}
\cA_e(z)&:=\big\{y\in \Real^n\ |\ \big((z,y)\in \cX(\cS)\ \text{and}\ (z-y)\cdot \bn(z)> 0\big)\\
&\hspace{1in}\text{or } \big((z,y)\in \cX(\cS)^c\ \text{and}\ (z-y)\cdot \bn(z)< 0\big)\big\},\\
\cA_i(z)&:=\big\{y\in \Real^n\ |\ \big((z,y)\in \cX(\cS)^c\ \text{and}\ (z-y)\cdot \bn(z)> 0\big)\\
&\hspace{1in}\text{or } \big((z,y)\in \cX(\cS)\ \text{and}\ (z-y)\cdot \bn(z)< 0\big)\big\}.
\end{align*}
As with the integral in \eqref{Hsper}, the integral in \eqref{Hs} must be understood in the principal value sense.  Note that this definition does not require the surface $\cS$ to be compact, so it could be unbounded; for a depiction of the sets $\cA_e(z)$ and $\cA_i(z)$, see Figure~\ref{AIAEcolor}.
\begin{figure}[h]
\centering
\includegraphics[width=4in]{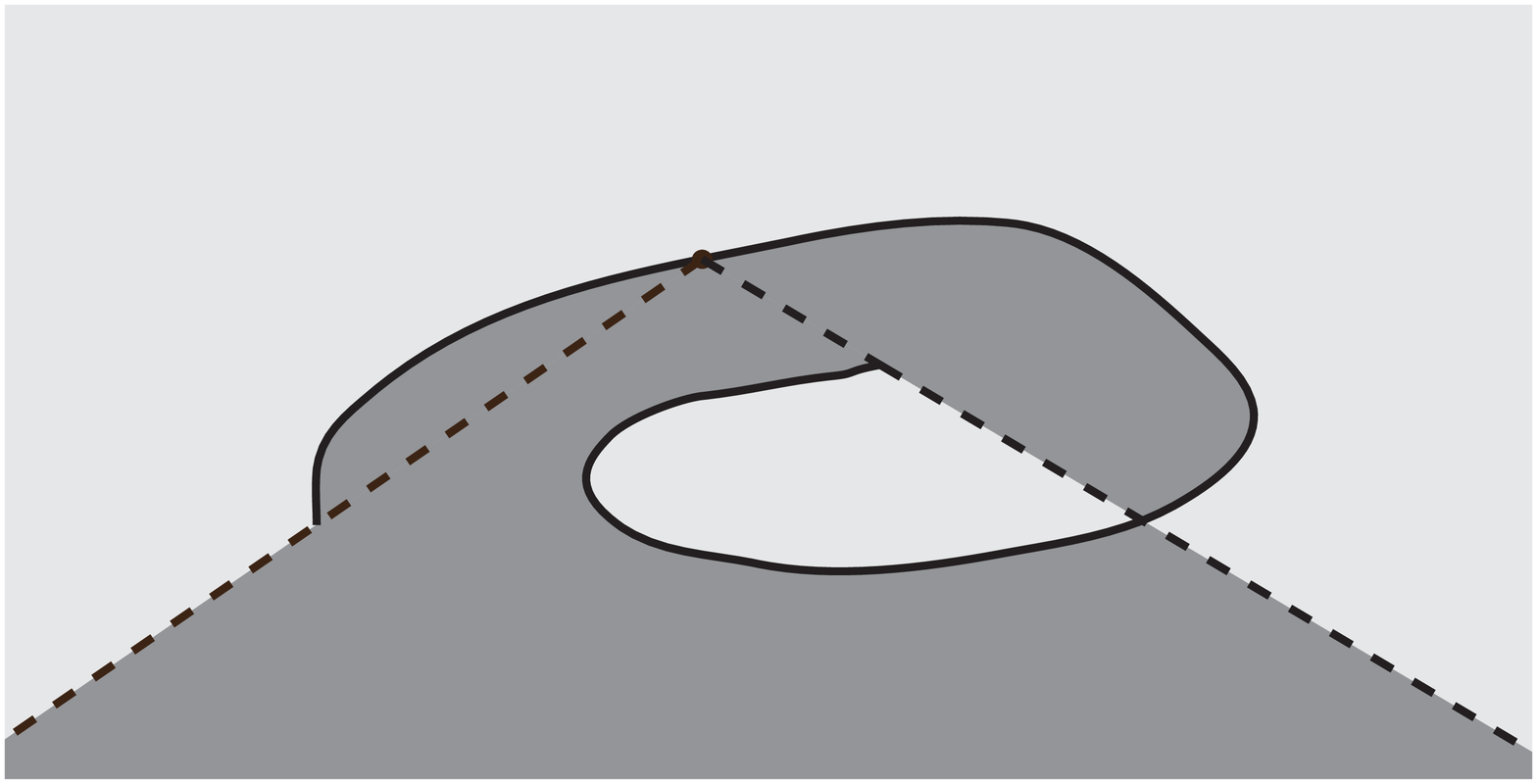}
\thicklines
\put(-170,100){$z$}
\put(-250,100){$\cA_e(z)$}
\put(-150,20){$\cA_i(z)$}
\put(-158,110){$\bn$}
\put(-80,100){$\cS$}
\put(-171.5,105){\rotatebox[origin=c]{100}{$\vector(1,0){25}$}}
\caption{The solid line depicts $\cS$. The set $\cA_e(z)$ is shown in light grey,  the set $\cA_i(z)$ in dark grey.  The dashed lines depict the part of the boundary between $\cA_e(z)$ and 
$\cA_i(z)$ that is not part of $\cS$.}
\label{AIAEcolor}
\end{figure}
\noindent On making use of the identity
\beqn
|z-y|^{-(n+\sigma)}=\frac{1}{\sigma} \text{div}_y\big[ |z-y|^{-(n+\sigma)}(z-y)\big]
\eeqn
(see \cite{Cetal}) and the divergence theorem, it is possible to write the nonlocal mean curvature as an integral over $\cS$:
\beqn\label{Hsrep}
H_\sigma(z)=\frac{2}{\sigma\omega_{n-2}} \int_\cS \frac{(z-y)\cdot \bn_{\cA_i(z)}(y)}{|z-y|^{n+\sigma}}\, dy.
\eeqn

In \cite{AV}, Abatangelo and Valdinoci introduced a notion of nonlocal directional curvature for surfaces being the  complete boundaries of sets. Their definition was extended to surfaces with boundary in 
\cite{PPGS}, in the way describe here below.  

For any point $z\in \cS$ and any unit vector $\be$ tangent to $\cS$ at $z$, let
\[\label{pixe}
\pi(z,\be):=\{y\in\Real^n\,|\;y=z+\rho\be+h\bn(z),\;\,\rho>0,\;\,h\in\Real\}
\]
be the half-plane through $z$ defined by the unit vector $\be$ and the normal $\bn(z)$ (see Figure \ref{fig2}). 
\begin{figure}[h]
\centering%
\includegraphics[height=6cm]{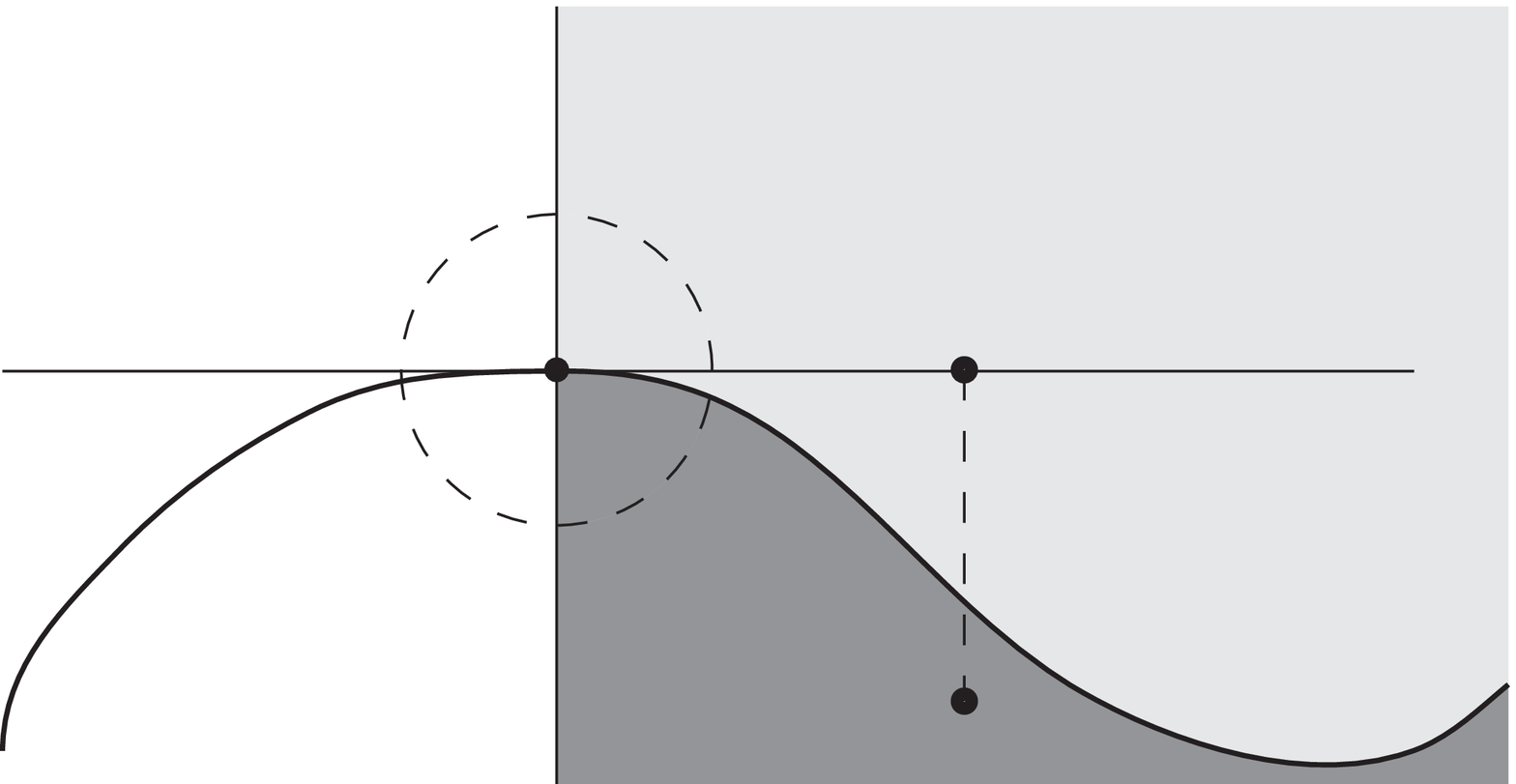}
\thicklines
\put(-200,20){$\cA_i(z)\cap \pi(z,\be)$}
\put(-115,10){$y$}
\put(-115,100){$y^\prime$}
\put(-205,95){$z$}
\put(-150,140){$\cA_e(z)\cap\pi(z,\be)$}
\put(-240,100){$B_\ve(z)$}
\put(-219.5,140){\rotatebox[origin=c]{90}{$\vector(1,0){25}$}}
\put(-235,140){$\bn(z)$}
\put(-80,90.5){\rotatebox[origin=c]{0}{$\vector(1,0){25}$}}
\put(-75,80){$\be$}
\caption{A depiction of the half plane $\pi(z,\be)$.}
\label{fig2}
\end{figure}
Moreover, given any point $y\in \pi(z,\be)$, let $y'$ denote the projection of this point onto the line through $z$ in the direction of $\be$. The  {\it{nonlocal directional curvature}} at $z$ in the direction of 
$\be$,
\beqn\label{Kse}
k_{\sigma,\be}(z):=\int_{\pi(z,\be)}\frac{|z-y'|^{n-2}\,\widehat\chi_\cS(z,y)}{|z-y|^{n+\sigma}} dy,
\eeqn
satisfies a limit relation analogous to \eqref{limHs}, namely
\beqn\label{limks}
\lim_{\sigma\rightarrow 1^-} (1-\sigma)k_{\sigma,\be}(z)=k_{\be}(z).
\eeqn
The nonlocal mean and directional curvatures at $z$ are related through
\beqn\label{Kseave}
H_\sigma(z)=\frac{1}{\omega_{n-2}} \int_{\cU(\cT_z\cS)}k_{\sigma,\be}(z)\, d\be,
\eeqn
where $\cU(\cT_z\cS)$ is the set of unit vectors in the tangent space $T_z\cS$ to $\cS$ at $z$.  This means that the nonlocal mean curvature is the average of the nonlocal directional curvatures, just as the local mean curvature is the average of the local directional curvatures.


\section{A notion of nonlocal curvature tensor}\label{section3}

Our definition of a nonlocal curvature tensor will be motivated by a new representation, that we now derive, for its local counterpart.

The curvature tensor $\bL(z)$ at a point $z\in\cS$ is defined as minus the surface gradient of the normal $\bn$ at $z$:
\beqn\label{clagrad}
\bL(z)=-^{s}\nabla \bn(z).
\eeqn
Given a unit vector $\be\in\cT_z\cS$, the directional curvature in the direction of $\be$ is
\beqn\label{dcv}
k_{\be}(z):=\be\cdot\bL(z)\be.
\eeqn
The following result shows that the curvature tensor is completely determined by the set of the directional curvatures.
\begin{proposition}
For all $z\in\cS$,
\beqn\label{Lclass}
\bL(z)=\frac{n-1}{2\omega_{n-2}} \int_{\cU(\cT_z\cS)}k_{\be}(z)\big((n+1)\be\otimes \be - \text{\bf 1}_{T_z\cS}\big)   \, d\be.
\eeqn
\end{proposition}

\begin{proof}
Let $\cB(T_z\cS)$ denote the unit ball in $T_z\cS$, and let   $\bx=r\be$, with $r\in [0,1]$ and $\be\in \cU(T_z\cS)$, denote a typical point of $\cB(T_z\cS)$. Using spherical coordinates we find that
\begin{align}
\int_{\cB(\cT_z\cS)}\bx\otimes \bx\, d\bx &= \int_{\cU(T_z\cS)}\int_0^1  r^{n-2} r \be\otimes r\be\,dr\, d\be= \frac{1}{n+1}  \int_{\cU(T_z\cS)}\be\otimes \be\,d\be\label{intee0}\\
&=\frac{1}{n+1}  \int_{\cB(T_z\cS)}\nabla \bx\,d\bx= \frac{\alpha_{n-1}}{n+1}\text{\bf 1}_{T_z\cS},\label{intee}
\end{align}
where the first identity of the second line is obtained by observing that $\cU(T_z\cS)$ can be identified as the set of unit normals to $\partial \cB(T_z\cS)$ and by applying the divergence theorem in the form
\[
\int_{\cB(T_z\cS)}\nabla \bx\,d\bx=\int_{\partial \cB(T_z\cS)}\bn\otimes\bn\,d\bn.
\]

Next, a use of \eqref{dcv} and \eqref{intee} and an application of the divergence theorem give 
\begin{align}
\nonumber\int_{\cU(T_z\cS)} k_\be(z)\, \be\otimes\be\, d\be &= \int_{\cU(T_z\cS)} (\be\otimes\bL(z)\be) \be\otimes\be\, d\be\\
\nonumber&= \int_{\cB(T_z\cS)} \nabla[(\bx\otimes\bL(z)\bx) \bx]\, d\bx\\
\nonumber&= \int_{\cB(T_z\cS)} [(\bx\cdot\bL(z)\bx) \text{\bf 1}_{T_z\cS} + 2\bx\otimes\bL(z)\bx\, ]d\bx\\
\nonumber&=\Big(\int_{\cB(T_z\cS)} \bx\otimes\bx \,d\bx\cdot \bL(z)  \Big) \text{\bf 1}_{T_z\cS}+2\Big(\int_{\cB(T_z\cS)} \bx\otimes\bx \,d\bx\Big)\bL(z)\\
&=\frac{\alpha_{n-1}}{n+1}\Big( (\tr \bL(z))\text{\bf 1}_{T_z\cS}+2\bL(z)\Big).
\end{align}
Solving for $\bL(z)$ shows that
\beqn\label{26}
\bL(z)=\frac{n+1}{2\alpha_{n-1}}\int_{\cU(T_z\cS)} k_\be(z) \be\otimes\be\, d\be-\frac{1}{2}(\tr\, \bL(z))\text{\bf 1}_{T_z\cS}.
\eeqn
Taking the trace of \eqref{26}, we find
\beqn \label{27}
\tr\,\bL(z)=\frac{1}{\alpha_{n-1}} \int_{\cU(T_z\cS)} k_\be(z)\, d\be;
\eeqn
the desired result follows  by substituting \eqref{27} in \eqref{26} and using the fact that
$\omega_{n-2}=(n-1)\alpha_{n-1}.$
\end{proof}

The definition of the nonlocal curvature tensor we propose is modelled on \eqref{Lclass}, with
\begin{enumerate}
\item[(i)] the directional curvature $k_\be(z)$ replaced by the nonlocal analog $k_{\sigma,\be}(z)$ defined by \eqref{Kse};
\item[(ii)] the $2$ appearing in the denominator replaced by $1+\sigma$;
\item[(iii)] the $n+1$ appearing in the integrand replaced  by $n+\sigma$.
\end{enumerate}  
Precisely, we define the nonlocal curvature tensor as
\beqn\label{Ldd}
\bL_\sigma(z):=\frac{n-1}{(1+\sigma)\omega_{n-2}} \int_{\cU(\cT_z\cS)}k_{\sigma,\be}(z)\big((n+\sigma)\be\otimes \be - \text{\bf 1}_{T_z\cS}\big)   \, d\be.
\eeqn
Notice that $\bL_\sigma(z)$ is a symmetric tensor which takes tangent vectors into tangent vectors, just like the classical curvature tensor does, moreover, it follows from \eqref{limks} and \eqref{Lclass} that
\beqn\label{limLs}
\lim_{\sigma\rightarrow 1^-}(1-\sigma)\bL_\sigma(z)=\bL(z).
\eeqn
By taking the trace of \eqref{Ldd} and making use of  \eqref{Kseave} we arrive at  
\beqn\label{trLs}
H_\sigma(z)=\frac{1}{n-1}\tr\, \bL_\sigma(z),
\eeqn
that is, the  nonlocal mean curvature is the mean of the terms on the diagonal appearing in any matrix representing, with respect to an orthonormal basis, the nonlocal curvature tensor. This identity would not hold as stated without the modification above indicated by (ii). The reason for the modification indicated by  (iii) will be explained in Section \ref{section4}.

Finally, we define the nonlocal Gaussian curvature at $z\in\cS$  as
\beqn\label{sigmagauss}
K_\sigma(z):=\text{det}\, \bL_\sigma(z).
\eeqn

We conclude this section with two results related to the nonlocal curvature tensor we just introduced. 
\vskip 6pt
\noindent(i)  $\bL_\sigma(z)$ could be written in terms of $\widehat\chi(z,\cdot)$ rather than $k_{\sigma,\be}(z)$.  Indeed, on using \eqref{Kse} and a change of variables, $\bL_\sigma(z)$ can be expressed as
\beqn\label{sLmod}
\bL_\sigma(z)=\frac{n-1}{(1+\sigma)\omega_{n-2}} \int_{\Real^n} \frac{\widehat\chi_\cS (z,y)}{|z-y|^{n+\sigma}}((n+\sigma) \hat\be_z(y)\otimes \hat\be_z(y)-\text{\bf 1}_{T_z\cS})\, dy,
\eeqn
where
\beqn
\hat\be_z(y):=\frac{y'-z}{|y'-z|}.
\eeqn
Using the fact that
\beqn
(\nabla_y \hat\be_z(y)) (z-y)=\textbf{0}\quad \text{for all }y\in\Real^n\backslash\{z\}
\eeqn
and arguments similar to those leading to \eqref{Hsrep}, one can show that $\bL_\sigma(z)$ also has the representation
\beqn\label{Lrep}
\bL_\sigma(z)=\frac{2(n-1)}{{ (1+\sigma)\sigma}\,\omega_{n-2}} \int_\cS \frac{(z-y)\cdot\bn_{\cA_i(z)}(y)}{|z-y|^{n+\sigma}}(( n+\sigma) \hat\be_z(y)\otimes \hat\be_z(y)-\text{\bf 1}_{T_z\cS}) \, dy
\eeqn
involving an integral over the surface $\cS$.

\vskip 6pt
\noindent(ii) 
Consider the case $n=3$. Then, \eqref{Ldd} gives 
\beqn\label{Ldd3}
\bL_\sigma(z)=\frac{1}{(1+\sigma)\pi} \int_{\cU(T_z\cS)}k_{\sigma,\be}(z)\big((3+\sigma)\be\otimes \be - \text{\bf 1}_{T_z\cS}\big)   \, d\be.
\eeqn
 
For any unit vector $\be\in\cT_z\cS$, let
$$
\be^\perp:=\bn(z)\times \be;
$$
note that $\be\cdot \tilde \be^\perp=- \be^\perp\cdot \tilde \be$, for any unit vectors $\be$ and $\tilde \be$.
Furthermore, a calculation shows that, for any constant $c$ and any unit vector $\be$,
$$
\text{cof}(c\, \be\otimes \be -\text{\bf 1}_{T_z\cS})= c\, \be^\perp\otimes \be^\perp-\text{\bf 1}_{T_z\cS},
$$
where ``cof'' denotes the cofactor of a tensor, i.e., the transpose of its adjugate. Since 
$$(\text{det}\,\bL_\sigma(z))\textbf{1}_{T_z\cS}=\bL_\sigma(z) \,(\text{cof}\,\bL_\sigma(z))^T,$$
we have that
\begin{align*}
2\, \text{det}\,\bL_\sigma &= \bL_\sigma \cdot \text{cof}\,\bL_\sigma \\
&=\bL_\sigma\cdot \text{cof}\,
\int_{\cU(T_z\cS)}\bar k_{\sigma,\be}(z)\big((3+\sigma)\be\otimes \be - \text{\bf 1}_{T_z\cS}\big)\, d\be\qquad \Big(\bar k_{\sigma,\be}:=\displaystyle\frac{1}{(1+\sigma)\pi}k_{\sigma,\be}\Big)\\
&= \bL_\sigma\cdot \,
 \int_{\cU(T_z\cS)}\bar k_{\sigma,\be}(z)\big((3+\sigma)\be^\perp\otimes \be^\perp - \text{\bf 1}_{T_z\cS}\big)\, d\be,
\end{align*}
where the last equality follows from the linearity of the cofactor map in dimension 2. Next, by the use of \eqref{Ldd3}, we find that
\begin{align*}
2\, \text{det}\,\bL_\sigma &=\int_{\cU(T_z\cS)}\bar k_{\sigma,\tilde\be}\big((3+\sigma)\tilde\be\otimes \tilde\be - \text{\bf 1}_{T_z\cS}\big)\, 
d\tilde\be\cdot 
 \int_{\cU(T_z\cS)}\bar k_{\sigma,\be}\big((3+\sigma)\be^\perp\otimes \be^\perp - \text{\bf 1}_{T_z\cS}\big)\, d\be\\
&=\int_{\cU(T_z\cS)}\int_{\cU(T_z\cS)} \bar k_{\sigma,\tilde\be}\bar k_{\sigma,\be}
\big[(3+\sigma)^2(\tilde\be\cdot \be^\perp)^2-2(3+\sigma)+2\big]\, d\be d\tilde\be\\
&=\int_{\cU(\cT_z\cS)}\int_{\cU(T_z\cS)} \bar k_{\sigma,\tilde\be}\bar k_{\sigma,\be}
\big[(3+\sigma)^2(\tilde\be\cdot \be^\perp)^2-2(2+\sigma)\big]\, d\be d\tilde\be.
\end{align*}

\noindent
Finally, for $\vartheta(\be,\tilde \be)$ the angle between the vectors $\be$ and $\tilde \be$,
$$(\tilde\be\cdot \be^\perp)^2=((\bn\times \be)\cdot \tilde\be)^2=((\be \times \tilde \be) \cdot \bn)^2=|\be \times \tilde \be|^2=\sin^2\vartheta(\be,\tilde \be).$$
\noindent
In view of definition \eqref{sigmagauss}, we conclude that
$$
K_\sigma(z) =\frac{1}{2{ (1+\sigma)^2}\pi^2}
\int_{\cU(\cT_z\cS)}\int_{\cU(\cT_z\cS)}  k_{\sigma,\tilde\be}(z) k_{\sigma,\be}(z)
\big[(3+\sigma)^2\sin^2\vartheta(\be,\tilde \be)-2(2+\sigma)\big]\, d\be d\tilde\be.
$$

\subsection{The case of a sphere}

Our present goal is  to compute the nonlocal curvature tensor of a sphere $\cS$ of radius $\rho$ 
in $\Real^n$; without loss of generality we take $\cS$ centered at the origin, as depicted in Fig. \ref{fig3}, the version of Fig. \ref{fig2} appropriate to the present context.
\begin{figure}[h]
\center%
\includegraphics[scale=0.5]{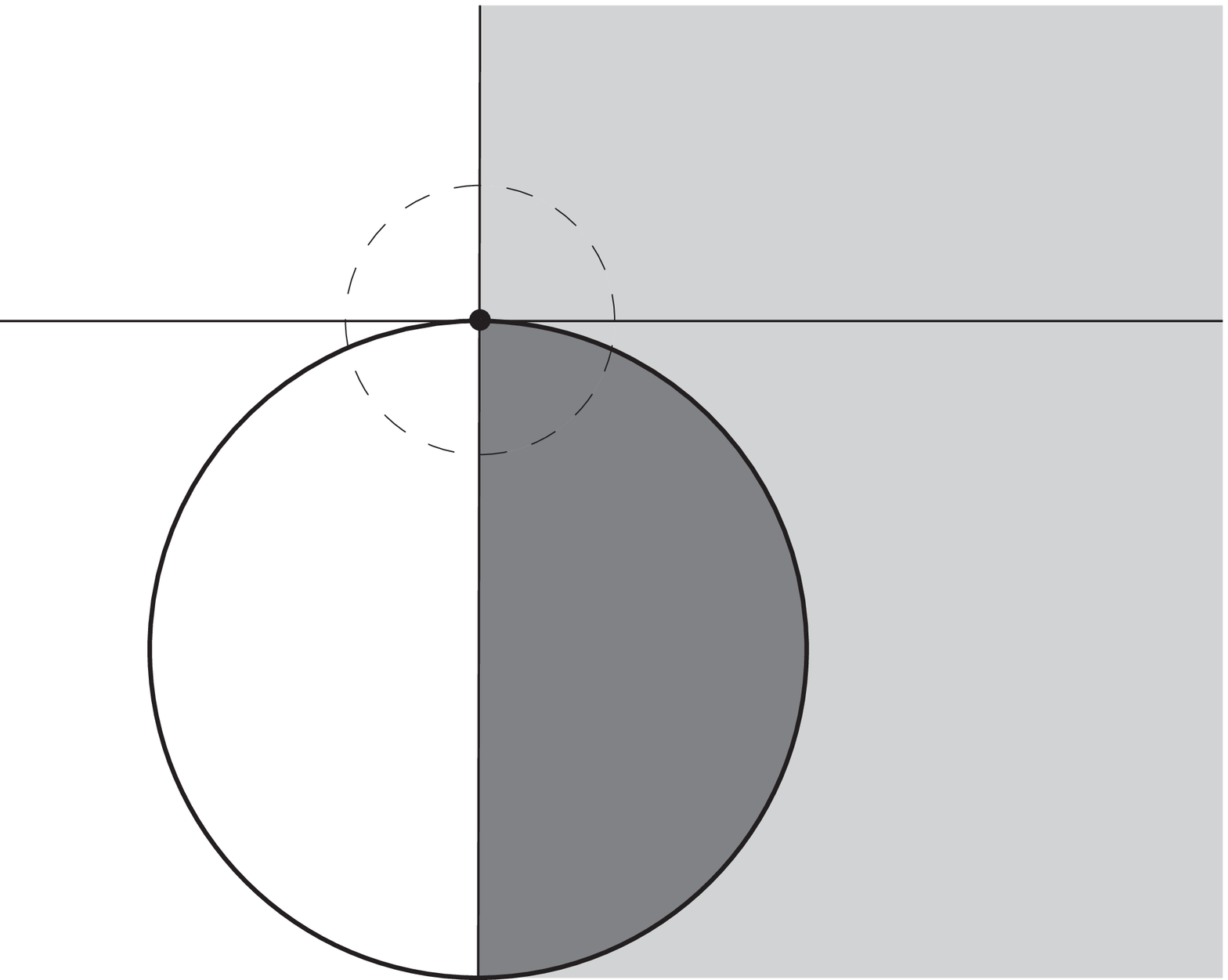}
\thicklines
\put(-167,70){$\cA_i(z)\cap \pi(z,\be)$}
\put(-165,155){$z$}
\put(-130,180){$\cA_e(z)\cap\pi(z,\be)$}
\put(-198,155){$B_\ve(z)$}
\put(-179,195){\rotatebox[origin=c]{90}{$\vector(1,0){25}$}}
\put(-165,195){$\bn(z)$}
\put(-80,149.8){\rotatebox[origin=c]{0}{$\vector(1,0){25}$}}
\put(-75,140){$\be$}
\caption{A depiction of the half plane $\pi(z,\be)$ when $\cS$ is a sphere.}
\label{fig3}
\end{figure}

In view of the symmetry intrinsic to a sphere, we specialize definition \eqref{Ldd}
\begin{align}\label{Ldd1}
\bL_\sigma(z)&=k_{\sigma,\be}(z)\frac{n-1}{(1+\sigma)\omega_{n-2}} \int_{\cU(\cT_z\cS)}\big((n+\sigma)\tilde \be\otimes\tilde \be - \text{\bf 1}_{T_z\cS}\big)   \, d\tilde\be&\\&=k_{\sigma,\be}(z)\text{\bf 1}_{T_z\cS},\label{Ldd2}
\end{align}
where the second equality follows from \eqref{intee0} and \eqref{intee}.
We now concentrate on computing the nonlocal directional derivative according to definition \eqref{Kse}, that is, on computing
\beqn
 k_{\sigma,\be}(z)=\lim_{\ve\rightarrow 0} \int_{\pi(z,\be)\setminus B_\varepsilon(z)} \frac{|y'-z|^{n-2}\widehat\chi_\cS(y,z)}{|y-z|^{n+\sigma}}\,dy.
\eeqn

 We begin by observing that
\begin{eqnarray}
 & &\int_{\pi(z,\be)\setminus B_\varepsilon(z)}   \frac{|y'-z|^{n-2}\widehat\chi_\cS(y,z)}{|y-z|^{n+\sigma}}\,dy \nonumber \\ 
 & &\qquad = \int_{(\pi(z,\be)\setminus B_\varepsilon(z))\cap\cA_i(z)} \frac{|y'-z|^{n-2}}{|y-z|^{n+\sigma}}\,dy-\int_{(\pi(z,\be)\setminus B_\varepsilon(z))\cap\cA_e(z)} \frac{|y'-z|^{n-2}}{|y-z|^{n+\sigma}}\,dy\nonumber\\ 
 & &\qquad =\int_{(\pi(z,\be)\setminus B_\varepsilon(z))\cap\cA_i(z)} \frac{|y'-z|^{n-2}}{|y-z|^{n+\sigma}}\,dy-\int_{(\pi(z,\be)\setminus B_\varepsilon(z))} \frac{|y'-z|^{n-2}}{|y-z|^{n+\sigma}}\,dy\nonumber\\
& &\qquad\qquad + \int_{(\pi(z,\be)\setminus B_\varepsilon(z))\cap\cA_e^c(z)} \frac{|y'-z|^{n-2}}{|y-z|^{n+\sigma}}\,dy\label{ddv2term}\\ 
 & &\qquad = -\int_{\pi(z,\be)\setminus B_\varepsilon(z)}\frac{|y'-z|^{n-2}}{|y-z|^{n+\sigma}}\,dy+2\int_{\pi(z,\be)\cap(\cA_i(z)\setminus B_\varepsilon(z))}\frac{|y'-z|^{n-2}}{|y-z|^{n+\sigma}}\,dy\,,\nonumber
\end{eqnarray}
where the first identity follows from \eqref{chiS},  $\cA_e^c(z)$   denotes the complement of $\cA_e(z)$, and the last identity follows since the set $(\pi(z,\be)\setminus B_\varepsilon(z))\cap (\cA_e^c(z)\setminus \cA_i(z))$ has measure zero.
As to the first of the two integrals on the right side, on setting
\beqn
y-z=r\ba(\varphi),\quad \ba(\varphi):=\cos\varphi\,\bn(z)+\sin\varphi\,\be,\quad (r,\varphi)\in (\varepsilon,\infty)\times (0,\pi),
\eeqn
we find that
\begin{align}
\nonumber\int_{\pi(z,\be)\setminus B_\varepsilon(z)} \frac{|y'-z|^{n-2}}{|y-z|^{n+\sigma}}\,dy&=\int_\ve^\infty \int_0^\pi r^{-1-\sigma}\sin^{n-2}\varphi\, d\varphi dr\\
\nonumber&=\frac{2}{\sigma\ve^\sigma}\int_0^{\pi/2} \sin^{n-2}(2\phi) d\phi\\
\nonumber&=\frac{2^{n-1}}{\sigma\ve^\sigma}\int_0^{\pi/2} \sin^{n-2}(\phi)\cos^{n-2}(\phi) d\phi\\
\label{firstint}&=\frac{2^{n-2}}{\sigma\ve^\sigma}B(\tfrac{n-1}{2},\tfrac{n-1}{2}),
\end{align}
where we have utilized \eqref{betatrig}, one of the many representations of the beta function.\footnote{For the reader's convenience, we have collected in the Appendix the definitions of the gamma and beta functions, as well as the properties of those functions we here use.} Using \eqref{gammahalf}, \eqref{gammaldf}, and \eqref{betagamma}, we find that \eqref{firstint} can be rewritten as
\begin{align}\label{firstint2}
\int_{\pi(z,\be)\setminus B_\varepsilon(z)} \frac{|y'-z|^{n-2}}{|y-z|^{n+\sigma}}\,dy&=\frac{2^{n-2}\Gamma(\tfrac{n-1}{2})^2}{\sigma\ve^\sigma\Gamma(n-1)}
=\frac{\Gamma(\tfrac{n-1}{2})\Gamma(\tfrac{1}{2})}{\sigma\ve^\sigma\Gamma(\tfrac{n}{2})}=\frac{1}{\sigma\ve^\sigma}B(\tfrac{n-1}{2},\tfrac{1}{2}).
\end{align}

The calculation of the second integral on the right-hand side of \eqref{ddv2term} is a bit more complicated. Note that, for $r\in(\varepsilon,2\rho)$, the spheres $B_r(z)$ and $\cS$ intersect in the plane $\pi(z,\be)$ at the point
\beqn
p(r)=z+r \sqrt{1-(\tfrac{r}{2\rho})^2}\be-\tfrac{r^2}{2\rho}\bn(z);
\eeqn
the angle between $(p(r)-z)$ and $\bn(z)$ is
\beqn
\varphi(r)=\pi-\arctan\frac{\sqrt{1-(\tfrac{r}{2\rho})^2}}{\tfrac{r}{2\rho}}\,;
\eeqn
in particular,
\beqn\label{phiprop}
\varphi(\varepsilon)=\pi-\arctan\frac{\sqrt{1-(\tfrac{\ve}{2\rho})^2}}{\tfrac{\ve}{2\rho}}\quad\text{and}\quad \lim_{\varepsilon\rightarrow 0} \varphi(\varepsilon)=\pi/2\,.
\eeqn
With this notation the second integral on the right-hand side of \eqref{ddv2term} can be written as
\begin{align}
\nonumber\int_{\pi(z,\be)\cap(\cA_i(z)\setminus B_\varepsilon(z))}\frac{|y'-z|^{2-n}}{|y-z|^{n+\sigma}}\,dy&=\int_{\varphi(\ve)}^\pi\int_\ve^{-2\rho\cos\phi} r^{-1-\sigma}\sin^{n-2}\phi\, drd\phi\\
\label{secondint}&\hspace{-.1in}=\frac{1}{\sigma}\int_{\varphi(\ve)}^\pi(\ve^{-\sigma}\sin^{n-2}\phi-(-2\rho\cos\phi)^{-\sigma}\sin^{n-2}\phi)\, d\phi.
\end{align}
Now{\clr ,} utilizing the change of variables $\phi\mapsto \phi-\pi/2$ and \eqref{betatrig} we find that
\begin{align}
\nonumber\int_{\pi(z,\be)\cap(\cA_i(z)\setminus B_\varepsilon(z))}\frac{|y'-z|^{2-n}}{|y-z|^{n+\sigma}}\,dy
&=\frac{1}{\sigma}\int_{\varphi(\ve)-\pi/2}^{\pi/2}(\ve^{-\sigma}\cos^{n-2}\phi-(2\rho\sin\phi)^{-\sigma}\cos^{n-2}\phi)\, d\phi\\
\nonumber&= \frac{B(\tfrac{1}{2},\tfrac{n-1}{2})}{2\sigma\ve^\sigma}+\frac{1}{\sigma\ve^\sigma} \int_{\phi(\ve)-\pi/2}^{0} \cos^{n-2}\phi\,d\phi\\
\label{secondint2}&\qquad\qquad-\frac{1}{\sigma}\int_{\phi(\ve)-\pi/2}^{\pi/2}(2\rho\sin\phi)^{-\sigma}\cos^{n-2}\phi\, d\phi\,.
\end{align}
Substituting \eqref{firstint2} and \eqref{secondint2} into \eqref{ddv2term} and using \eqref{betatrig} again results in
\beqn\label{bkse}
k_{s,\be}(z)=\lim_{\ve\rightarrow 0} \Big(\frac{2}{\sigma\ve^\sigma}\int^0_{\varphi(\ve)-\pi/2}\cos^{n-2}\varphi\, d\varphi-\frac{2}{\sigma}\int_{\phi(\ve)-\pi/2}^{\pi/2}(2\rho\sin\phi)^{-\sigma}\cos^{n-2}\phi\, d\phi\Big).
\eeqn
Using de L'H$\hat{\text{o}}$pital's rule, one finds that
\beqn
\lim_{\ve\rightarrow 0} \frac{\int^0_{\varphi(\ve)-\pi/2}\cos^{n-2}\varphi\, d\varphi}{\ve^\sigma}=\lim_{\ve\rightarrow 0} -\ve^{1-\sigma}\sin^{n-2}\varphi(\ve) \varphi'(\ve)=0.
\eeqn
Moreover, using \eqref{phiprop}$_2$ and \eqref{betatrig} we have
\beqn
\lim_{\ve\rightarrow 0} \frac{2}{\sigma}\int_{\phi(\ve)-\pi/2}^{\pi/2}(2\rho\sin\phi)^{-\sigma}\cos^{n-2}\phi\, d\phi
=\frac{B(\tfrac{1-\sigma}{2},\tfrac{n-1}{2})}{\sigma (2\rho)^\sigma}.
\eeqn
Thus, the nonlocal directional curvature vector for a sphere is
\beqn
k_{s,\be}(z)=-\frac{B(\tfrac{1-\sigma}{2},\tfrac{n-1}{2})}{\sigma(2\rho)^\sigma}.
\eeqn
In conclusion, in view of \eqref{Ldd2}, the nonlocal curvature tensor for a sphere is
\beqn
\bL_\sigma(z)=-\frac{B(\tfrac{1-\sigma}{2},\tfrac{n-1}{2})}{\sigma(2\rho)^\sigma}\bone_{T_z\cS}.
\eeqn

\section{Possible connections with fractional operators}\label{section4}

Here we outline a possible connection between the notion of nonlocal curvature tensor proposed in definition \eqref{Ldd} and certain fractional differential operators (see Subsection \ref{app1}).  This connection is motivated by the following lines of thinking.  

In the case where the surface under attention is the boundary of a set $E$, the classical curvature tensor is the opposite of the surface gradient of the normal (recall \eqref{clagrad}); moreover,  the normal can be obtained by looking at the jump part of the distributional derivative of the characteristic function of $E$.  Thus, the curvature tensor is related to a second derivative of $\chi_E$ or, equivalently, to the function $\tilde\chi_E$ defined in \eqref{chibar}. 
When the surface is not the boundary of a set, the appropriate generalization of $\tilde\chi_E$ is the function $\widehat\chi_\cS$ defined in \eqref{chiS}.  
If we compare \eqref{Hs}-\eqref{chiS} with \eqref{fraclaplace0}, we are driven to relate  the nonlocal mean curvature to the fractional Laplacian:
\beqn
H_\sigma(z)=\frac{1}{\nu_\alpha\omega_{n-2}}(-\Delta)^{\sigma/2}\widehat\chi_\cS(z,\cdot).
\eeqn
On using the identities \eqref{trdivid} and \eqref{divid} in Subsection \ref{app1}, it follows from this relation  that
\beqn
H_\sigma(z)=\frac{1}{\nu_\sigma\omega_{n-2}}\tr (\nabla^{\sigma/2}\nabla^{\sigma/2}\widehat\chi_\cS(z,\cdot)).
\eeqn
But this is not the only reason why we argue that perhaps a second fractional gradient of $\widehat\chi_\cS$ may be related to the nonlocal curvature tensor we propose. In fact, utilizing
 \eqref{sfdformula}, a result proved in Subsection \ref{app2},  we see that
\beqn\label{chiSsd}
\nabla^{\sigma/2}\nabla^{\sigma/2}\widehat\chi_\cS(z,\cdot)=\frac{-\nu_{\sigma}}{\sigma} \int_{\Real^n} \frac{\widehat\chi_\cS(z,z+\bv)}{|\bv|^{n+\sigma}}\Big(\frac{n+\sigma}{|\bv|^2}\bv\otimes\bv-\bone_{\Real^n}\Big)d\bv.
\eeqn
Let us compare the right side of \eqref{chiSsd} and \eqref{sLmod}, a relation we repeat here for the reader's convenience,
\[
\bL_\sigma(z)=\frac{n-1}{(1+\sigma)\omega_{n-2}} \int_{\Real^n} \frac{\widehat\chi_\cS (z,y)}{|z-y|^{n+\sigma}}((n+\sigma) \hat\be_z(y)\otimes \hat\be_z(y)-\text{\bf 1}_{T_z\cS})\, dy.
\]
The similarity is striking.\footnote{It was the $n+\sigma$ factor   appearing in front of $\bv\otimes\bv$ in the integrand in \eqref{chiSsd} which prompted us to introduce in \eqref{sLmod} the modification {\clr (}iii) stated in Section \ref{section3}.} However, 
unfortunately, although the trace of \eqref{chiSsd} does give the nonlocal mean curvature up to a multiplicative constant, it does not converge to the classical curvature tensor  in the limit as $\sigma$ goes to 1.  This negative outcome leaves us with a final, admittedly vague conjecture. While
the classical curvature tensor is defined using a surface gradient, to our knowledge no fractional analog of surface gradient has been developed yet.  Perhaps an appropriate notion of fractional surface gradient might establish that there is a direct connection with the nonlocal curvature tensor we have proposed.

\subsection{Fractional gradient, Laplacian, and divergence}
\label{app1}
For each $\alpha\in(0,1)$,  let 
\begin{equation}\label{mua}
\mu_\alpha:=\frac{2^\alpha \Gamma\left(\frac{n+\alpha+1}2\right)}{\pi^{n/2}\Gamma\left(\frac{1-\alpha}2\right)}\qquad\text{and}\qquad \nu_\alpha:=\frac{2^\alpha\Gamma(\frac{n+\alpha}{2})}{\pi^{n/2}\Gamma(-\frac{\alpha}2)}.
\end{equation}
The fractional gradient  $\nabla^\alpha$ and the Laplacian $(-\Delta)^{\alpha/2}$ of a function $f$ are
\begin{align}
\label{fracgrad0}\nabla^\alpha f(x)&:=\mu_\alpha\int_{\Real^n} \frac{f(y)\otimes(y-x)}{|x-y|^{n+\alpha+1}}dy,\\
\label{fraclaplace0}(-\Delta)^{\alpha/2} f(x)&:=\nu_\alpha\int_{\Real^n} \frac{f(y)-f(x)}{|x-y|^{n+\alpha}}dy.
\end{align}
Notice that these definitions make sense if $f$ is scalar-, vector-, or tensor-valued (in the first instance,  $f(y)\otimes(y-x)$ must be understood as $f(y)(y-x)$).  

The fractional divergence of a vector-valued function $\bw$ is 
\beqn
\label{fracdiv0}\text{div}^\alpha \bw(x):=\mu_\alpha\int_{\Real^n} \frac{(\bw(y)-\bw(x))\cdot(y-x)}{|x-y|^{n+\alpha+1}}dy.
\eeqn
One can readily see that
\beqn\label{trdivid}
\tr(\nabla^\alpha \bw)=\text{div}^\alpha \bw;
\eeqn
verifying the identity
\beqn\label{divid}
\text{div}^\alpha(\nabla^\beta f)=-(-\Delta)^{(\alpha+\beta)/2}f
\eeqn
is harder (for a proof, see 
\cite{Silhavy19}).

\subsection{A characterization of the fractional Hessian}\label{app2}

Here we characterize the second fractional gradient of a scalar-valued function.
The main result is the following theorem.

\begin{theorem}\label{hess}
Let $\alpha, \beta\in(0,1)$. If $f$ is a scalar-valued function, then
\beqn\label{sfdformula}
\nabla^\alpha(\nabla^\beta f)(x)=\frac{-\nu_{\alpha+\beta}}{(\alpha+\beta)}\int_{\Real^n}  \frac{f(x+\bv)-f(x)}{|\bv|^{n+\alpha+\beta}}
  \Big( \frac{n+\alpha+\beta}{|\bv|^{2}} \bv\otimes\bv - \textbf{\bf 1}_n \Big)
 d\bv.
\eeqn
\end{theorem}

To prove it, we make use of the Fourier transform, that for a function $f$ is defined by
\begin{equation}\label{FT}
\mathcal{F}f(\xi)=\hat f(\xi)=\int_{\mathbb{R}^n}e^{-2\pi i x\cdot \xi} f(x) \,dx.
\end{equation}
The following identities are classical, see for instance \cite{Stein},
\begin{equation}\label{IFT}
f(x)=\int_{\mathbb{R}^n}e^{2\pi i x\cdot \xi} \hat f(\xi)\,d\xi,
\end{equation}

\begin{equation}\label{FTd}
\mathcal{F}(\partial^\gamma f)(\xi)=(2\pi i \xi)^\gamma \hat f(\xi),
\end{equation}

\begin{equation}\label{dFT}
\partial^\gamma \hat f(\xi)=\mathcal{F}\big((-2\pi i x)^\gamma f(x)\big)(\xi),
\end{equation}
where $\gamma$ is a multi-index.

The Gauss--Weierstrass kernel defined as
\beqn
g_t(\bu)=\frac{1}{(4\pi t)^{n/2}}e^{-\frac{|\bu|^2}{4t}},
\eeqn
satisfies the following identities:
\begin{equation}\label{FTgw}
g_t(\bu)=\mathcal{F}\big(e^{-4\pi^2 t |x|^2}\big)(\bu),
\end{equation}
and
\beqn\label{GWprop0}
\frac{1}{|\bu|^{n+\alpha+1}}=\frac{\pi^{n/2}}{2^{\alpha+1}\Gamma\left(\frac{n+\alpha+1}2\right)}\int_0^\infty g_t(\bu)t^{-(\alpha+3)/2}dt=\frac{1}{2\mu_\alpha\Gamma(\frac{1-\alpha}{2})}\int_0^\infty g_t(\bu)t^{-(\alpha+3)/2}dt,
\eeqn
where $\Gamma$ denotes the gamma function defined in \eqref{defgamma}.

The proof of Theorem \ref{hess} will be achieved through two Lemmas.

\begin{lemma}\label{lemGst}
Let $\gamma$ be a multi-index and
$$
G^\gamma_{s,t}(\bv)=\int_{\Real^n} \bu^\gamma g_s(\bv-\bu) g_t(\bu)\,d\bu.
$$
Then,
if $|\gamma|=1$ or  $|\gamma|=2$ and $\max_k{\gamma_k}=1$
 $$
G^\gamma_{s,t}(\bv)=\Big(\frac{t \bv}{s+t}\Big)^\gamma g_{s+t}(\bv),
$$
while if $|\gamma|=2$ and $\max_k{\gamma_k}=2$
 $$
G^\gamma_{s,t}(\bv)=\Big(\big(\frac{t \bv}{s+t}\big)^\gamma+\frac{2 s t}{s+t}\Big) g_{s+t}(\bv).
$$
\end{lemma}

\begin{proof}
By using \eqref{FT} and \eqref{FTgw} we have
$$
\begin{aligned}
G^\gamma_{s,t}(\bv)&=\int_{\Real^n} \bu^\gamma g_t(\bu)  \int_{\Real^n}  e^{-2\pi i x\cdot (\bv-\bu)} e^{-4\pi^2 s |x|^2} \,dx  \,d\bu\\
&=\frac 1{(2\pi i)^\gamma}\int_{\Real^n} e^{-2\pi i x\cdot \bv} e^{-4\pi^2 s |x|^2}    \int_{\Real^n}  e^{2\pi i x\cdot \bu} (2\pi i \bu)^\gamma g_t(\bu)   \,d\bu \,dx
\end{aligned}
$$
and thanks to \eqref{FTd} and \eqref{FTgw} we find
$$
\begin{aligned}
G^\gamma_{s,t}(\bv)
&=\frac 1{(2\pi i)^\gamma}\int_{\Real^n} e^{-2\pi i x\cdot \bv} e^{-4\pi^2 s |x|^2}    \int_{\Real^n}  e^{2\pi i x\cdot \bu} \mathcal{F}\big(\partial^\gamma e^{-4\pi^2 t |x|^2}\big)(\bu)   \,d\bu \,dx\\
&=\frac 1{(2\pi i)^\gamma}\int_{\Real^n} e^{-2\pi i x\cdot \bv} e^{-4\pi^2 s |x|^2}   \partial^\gamma e^{-4\pi^2 t |x|^2}\,dx
\end{aligned}
$$
where the last equality follows from \eqref{IFT}.
If $|\gamma|=1$,
\begin{equation}\label{g1}
 \partial^\gamma e^{-4\pi^2 t |x|^2}=-8\pi^2 t x^\gamma e^{-4\pi^2 t |x|^2},
\end{equation}
if $|\gamma|=2$ and $\max_k{\gamma_k}=1$, {\it i.e.}, $\gamma=(0,\ldots ,0,1,0,\ldots,0,1,0,\ldots,0)$, then
\begin{equation}\label{g2d}
 \partial^\gamma e^{-4\pi^2 t |x|^2}=(-8\pi^2 t x)^\gamma e^{-4\pi^2 t |x|^2},
\end{equation}
while if $|\gamma|=2$ and $\max_k{\gamma_k}=2$, {\it i.e.}, $\gamma=(0,\ldots ,0,2,0,\ldots,0)$, then
\begin{equation}\label{g2u}
 \partial^\gamma e^{-4\pi^2 t |x|^2}=(-8\pi^2 t x)^\gamma e^{-4\pi^2 t |x|^2}-8\pi^2 t  e^{-4\pi^2 t |x|^2}.
\end{equation}
For $|\gamma|=1$ or $|\gamma|=2$ and $\max_k{\gamma_k}=1$ we find
\begin{align}
G^\gamma_{s,t}(\bv)
&=\frac 1{(2\pi i)^\gamma}\int_{\Real^n} e^{-2\pi i x\cdot \bv} e^{-4\pi^2 (s+t) |x|^2}  (-8\pi^2 t x)^\gamma\,dx \nonumber\\
&=\frac 1{(2\pi i)^\gamma}\Big(\frac t{s+t}\Big)^\gamma\int_{\Real^n} e^{-2\pi i x\cdot \bv} e^{-4\pi^2 (s+t) |x|^2}  (-8\pi^2 (s+t) x)^\gamma\,dx\nonumber\\
&=\frac 1{(2\pi i)^\gamma}\Big(\frac t{s+t}\Big)^\gamma\int_{\Real^n} e^{-2\pi i x\cdot \bv}  \partial^\gamma e^{-4\pi^2 (s+t) |x|^2}\,dx\label{gg00}
\end{align}
where we used \eqref{g1} or \eqref{g2d}, according to $\gamma$. Hence, from\eqref{FT}, \eqref{FTd}, and \eqref{FTgw} it follows that
\begin{align}
G^\gamma_{s,t}(\bv)
&=\frac 1{(2\pi i)^\gamma}\Big(\frac t{s+t}\Big)^\gamma\mathcal{F}\big(\partial^\gamma  e^{-4\pi^2 (s+t) |x|^2}\big)(\bv)\nonumber\\
&=\frac 1{(2\pi i)^\gamma}\Big(\frac t{s+t}\Big)^\gamma (2\pi i \bv)^\gamma\mathcal{F}\big(e^{-4\pi^2 (s+t) |x|^2}\big)(\bv)\nonumber\\
&=\Big(\frac{t \bv}{s+t}\Big)^\gamma g_{s+t}(\bv). \label{gg01}
\end{align}
Instead, if $|\gamma|=2$ and $\max_k{\gamma_k}=2$ we have
$$
\begin{aligned}
G^\gamma_{s,t}(\bv)
&=\frac 1{(2\pi i)^\gamma}\int_{\Real^n} e^{-2\pi i x\cdot \bv} e^{-4\pi^2 (s+t) |x|^2}  \big((-8\pi^2 t x)^\gamma-8\pi^2 t\big)\,dx.
\end{aligned}
$$
By the identity
$$
\begin{aligned}
(-8\pi^2 t x)^\gamma-8\pi^2 t&= \Big(\frac{t}{s+t}\Big)^2\big((-8\pi^2 (s+t) x)^\gamma-8\pi^2 (s+t)\big)-\frac{8\pi^2 s t}{s+t}\\
&= \Big(\frac{t}{s+t}\Big)^2\frac{ \partial^\gamma e^{-4\pi^2 (s+t) |x|^2}}{e^{-4\pi^2 (s+t) |x|^2}}-\frac{8\pi^2 s t}{s+t}
\end{aligned}
$$
we find
$$
\begin{aligned}
G^\gamma_{s,t}(\bv)
&=\frac 1{(2\pi i)^\gamma} \Big(\frac{t}{s+t}\Big)^2\int_{\Real^n} e^{-2\pi i x\cdot \bv} \partial^\gamma e^{-4\pi^2 (s+t) |x|^2}\,dx\\
&\qquad\qquad-\frac{8\pi^2 s t}{s+t}\frac 1{(2\pi i)^\gamma}\int_{\Real^n} e^{-2\pi i x\cdot \bv} e^{-4\pi^2 (s+t) |x|^2}  \,dx.
\end{aligned}
$$
The first line of this equation is equal to \eqref{gg00} and hence to \eqref{gg01}, while the second may be simplified by means of \eqref{FTgw}. It is found that
 $$
G^\gamma_{s,t}(\bv)=\Big(\frac{t \bv}{s+t}\Big)^\gamma g_{s+t}(\bv)+\frac{2 s t}{s+t}\ g_{s+t}(\bv).
$$

\end{proof}

\begin{lemma}
We have
\beqn\label{sfg2}
 \int_{\Real^n} \frac{(\bv-\bu)\otimes\bu}{|\bu-\bv|^{n+\beta+1}|\bu|^{n+\alpha+1}}d\bu
 = \frac{-\nu_{\alpha+\beta}}{\mu_\alpha\mu_\beta (\alpha+\beta) |\bv|^{n+\alpha+\beta}}
  \Big( \frac{n+\alpha+\beta}{|\bv|^{2}} \bv\otimes\bv - \textbf{\bf 1}_{\Real^n} \Big).
\eeqn

\end{lemma}
\begin{proof}
With \eqref{GWprop0} we find
\begin{align}
 &{4\mu_\alpha\mu_\beta\Gamma(\tfrac{1-\alpha}{2})\Gamma(\tfrac{1-\beta}{2})}\int_{\Real^n} \frac{(\bv-\bu)\otimes\bu}{|\bu-\bv|^{n+\beta+1}|\bu|^{n+\alpha+1}}d\bu\nonumber\\
 & \hspace{2cm} =\int_{\Real^n} (\bv-\bu)\otimes\bu\int_0^\infty g_s(\bv-\bu)s^{-(\beta+3)/2}ds\int_0^\infty g_t(\bu)t^{-(\alpha+3)/2}dtd\bu\nonumber\\
  & \hspace{2cm} =  \int_0^\infty\int_0^\infty s^{-(\beta+3)/2} t^{-(\alpha+3)/2}
  \int_{\Real^n} (\bv-\bu)\otimes\bu\, g_s(\bv-\bu)g_t(\bu)d\bu ds dt \nonumber  \\
    & \hspace{2cm} =   \int_0^\infty\int_0^\infty s^{-(\beta+3)/2} t^{-(\alpha+3)/2}
  \bH(\bv, s,t) ds dt  \label{Heq}
\end{align}
where we have set
$$
\bH(\bv, s,t)= \int_{\Real^n} (\bv-\bu)\otimes\bu\, g_s(\bv-\bu)g_t(\bu)d\bu.
$$
The component $pq$, with $p\ne q$, is
$$
H_{pq}(\bv, s,t)= \int_{\Real^n} (v_p-u_p)u_q\, g_s(\bv-\bu)g_t(\bu)d\bu=\int_{\Real^n} (v_p \bu^{\bar \gamma}-\bu^{ \gamma})\, g_s(\bv-\bu)g_t(\bu)d\bu
$$
where
$$
\begin{aligned}
\bar \gamma&=(0,\ldots ,0,1,0,\ldots,0) \quad\mbox{ with the one in the $q$ position}\\
\gamma&=(0,\ldots ,0,1,0,\ldots,0,1,0,\ldots,0) \quad\mbox{ with the ones in the $p$ and $q$ positions}.
\end{aligned}
$$ 
By Lemma \ref{lemGst},  for $p\ne q$ we have
\begin{align}
H_{pq}(\bv, s,t)&= v_p G^{\bar \gamma}_{s,t}(\bv)-G^{\gamma}_{s,t}(\bv)
=v_p\Big(\frac{t \bv}{s+t}\Big)^{\bar \gamma} g_{s+t}(\bv)-\Big(\frac{t \bv}{s+t}\Big)^\gamma g_{s+t}(\bv)\nonumber\\
&=\frac{t}{s+t}v_p v_q g_{s+t}(\bv)-\big(\frac{t }{s+t}\big)^2 v_p v_q g_{s+t}(\bv)=\frac{st}{(s+t)^2}v_p v_q g_{s+t}(\bv),\label{gg02}
\end{align}
while the $qq$ component is
$$
H_{qq}(\bv, s,t)= \int_{\Real^n} (v_qu_p-u_q^2)\, g_s(\bv-\bu)g_t(\bu)d\bu=\int_{\Real^n} (v_p \bu^{\bar \gamma}-\bu^{ \hat \gamma})\, g_s(\bv-\bu)g_t(\bu)d\bu,
$$
where
$$
\hat \gamma=(0,\ldots ,0,2,0,\ldots,0) \quad\mbox{ with the two in the $q$ position}.
$$ 
Again, by Lemma \ref{lemGst},  we have
\begin{align}
H_{qq}(\bv, s,t)&=v_q G^{\bar \gamma}_{s,t}(\bv)-G^{\hat \gamma}_{s,t}(\bv)=v_q\Big(\frac{t \bv}{s+t}\Big)^{\bar \gamma} g_{s+t}(\bv)-\Big(\big(\frac{t \bv}{s+t}\big)^{\hat \gamma}+\frac{2 s t}{s+t}\Big) g_{s+t}(\bv)\nonumber\\
&=\frac{t}{s+t}v_q v_q g_{s+t}(\bv)-\big(\frac{t }{s+t}\big)^2 v_q v_q g_{s+t}(\bv)-\frac{2 s t}{s+t}g_{s+t}(\bv)\nonumber\\
&=\frac{st}{(s+t)^2}v_q v_q g_{s+t}(\bv)-\frac{2 s t}{s+t}g_{s+t}(\bv).\label{gg03}
\end{align}
By \eqref{gg02} and \eqref{gg03} we deduce that
$$
\bH(\bv, s,t)= \big(\frac{st}{(s+t)^2} \bv\otimes\bv -\frac{2 s t}{s+t}\textbf{\bf 1}_{\Real^n}\big)g_{s+t}(\bv).
$$
We therefore find
$$
\begin{aligned}
 & \int_0^\infty\int_0^\infty s^{-(\beta+3)/2} t^{-(\alpha+3)/2} \bH(\bv, s,t) ds dt  \\
 & \hspace{1cm} =   \int_0^\infty\int_0^\infty \frac{s^{-(\beta+1)/2} t^{-(\alpha+1)/2}}{s+t} \big(\frac{1}{s+t} \bv\otimes\bv -2\textbf{\bf 1}_{\Real^n}\big)g_{s+t}(\bv)ds dt.
 \end{aligned}
$$
Setting $p=s+t$ and $r=t/(s+t)$ we have that $s=p(1-r)$ and $t=p r$ and the Jacobian of the transformation is $1/(s+t)$. Thus
$$
\begin{aligned}
 & \int_0^\infty\int_0^\infty s^{-(\beta+3)/2} t^{-(\alpha+3)/2} \bH(\bv, s,t) ds dt  \\
 & \hspace{1cm} =   \int_0^\infty\int_0^1 (p(1-r))^{-(\beta+1)/2} (pr)^{-(\alpha+1)/2} \big(\frac{1}{p} \bv\otimes\bv -2\textbf{\bf 1}_{\Real^n}\big)g_{p}(\bv)dr dp\\
  & \hspace{1cm} =   \int_0^\infty p^{-(\alpha +\beta+2)/2}  \big(\frac{1}{p} \bv\otimes\bv -2\textbf{\bf 1}_{\Real^n}\big)g_{p}(\bv) dp\int_0^1 (1-r)^{-(\beta+1)/2} r^{-(\alpha+1)/2}dr\\
  & \hspace{1cm} =   B(\tfrac{1-\alpha}2,\tfrac{1-\beta}2)\int_0^\infty p^{-(\alpha +\beta+2)/2}  \big(\frac{1}{p} \bv\otimes\bv -2\textbf{\bf 1}_{\Real^n}\big)g_{p}(\bv) dp,
  \end{aligned}
$$
where to obtain the last identity we used \eqref{betatt}. 
By using \eqref{GWprop0}, we find that
$$
\begin{aligned}
 & \int_0^\infty\int_0^\infty s^{-(\beta+3)/2} t^{-(\alpha+3)/2} \bH(\bv, s,t) ds dt  \\
  & \hspace{1cm} = B(\tfrac{1-\alpha}2,\tfrac{1-\beta}2)\Big(
 \frac{2^{\alpha+\beta+2}\Gamma\left(\frac{n+\alpha+\beta+2}2\right)}{\pi^{n/2}|\bv|^{n+\alpha+\beta+2}} \bv\otimes\bv -
  2 \frac{2^{\alpha+\beta}\Gamma\left(\frac{n+\alpha+\beta}2\right)}{\pi^{n/2}|\bv|^{n+\alpha+\beta}} \textbf{\bf 1}_{\Real^n}
  \Big)\\
  & \hspace{1cm} = B(\tfrac{1-\alpha}2,\tfrac{1-\beta}2) 
   \frac{2^{\alpha+\beta+1}\Gamma\left(\frac{n+\alpha+\beta}2\right)}{\pi^{n/2}|\bv|^{n+\alpha+\beta}}
  \Big( \frac{n+\alpha+\beta}{|\bv|^{2}} \bv\otimes\bv - \textbf{\bf 1}_{\Real^n}\Big),
 \end{aligned}
$$
where we used \eqref{gammafact}. From this identity, \eqref{betagamma},
and \eqref{Heq} we deduce that
\begin{align}
 &\int_{\Real^n} \frac{(\bv-\bu)\otimes\bu}{|\bu-\bv|^{n+\beta+1}|\bu|^{n+\alpha+1}}d\bu\nonumber\\
    & \hspace{1cm} = \frac{1}{4\mu_\alpha\mu_\beta\Gamma(1-\frac{\alpha+\beta}{2})}
   \frac{2^{\alpha+\beta+1}\Gamma\left(\frac{n+\alpha+\beta}2\right)}{\pi^{n/2}|\bv|^{n+\alpha+\beta}}
  \Big( \frac{n+\alpha+\beta}{|\bv|^{2}} \bv\otimes\bv - \textbf{\bf 1}_{\Real^n} \Big)\\
      & \hspace{1cm} = \frac{-1}{\mu_\alpha\mu_\beta|\bv|^{n+\alpha+\beta}(\alpha+\beta)}
   \frac{2^{\alpha+\beta}\Gamma\left(\frac{n+\alpha+\beta}2\right)}{\pi^{n/2}\Gamma(-\frac{\alpha+\beta}{2})}
  \Big( \frac{n+\alpha+\beta}{|\bv|^{2}} \bv\otimes\bv - \textbf{\bf 1}_{\Real^n} \Big)
\end{align}
from which the statement of the Lemma follows.
\end{proof}

\begin{proof}[Proof of Theorem \ref{hess}]
By applying \eqref{fracgrad0} twice, we deduce that
\beqn
\nabla^\alpha(\nabla^\beta f)(x)=\mu_\alpha \mu_\beta\int_{\Real^n} \int_{\Real^n} \frac{f(z)(z-y)\otimes(y-x)}{|y-z|^{n+\beta+1}|x-y|^{n+\alpha+1}}dzdy,
\eeqn
and making the change of variables $\bu=y-x$ and $\bv=z-x$, we obtain
\beqn\label{sfg}
\nabla^\alpha(\nabla^\beta f)(x)=\mu_\alpha \mu_\beta\int_{\Real^n} f(x+\bv)\Big[ \int_{\Real^n} \frac{(\bv-\bu)\otimes\bu}{|\bu-\bv|^{n+\beta+1}|\bu|^{n+\alpha+1}}d\bu\Big] d\bv.
\eeqn
The proof of Theorem \ref{hess} readily follows  from \eqref{sfg} and \eqref{sfg2}.
\end{proof}


\section{Appendix. Some properties of the gamma and beta functions}
The gamma function $\Gamma$ is defined for all positive numbers $x$ by
\beqn\label{defgamma}
\Gamma(x)=\int_0^\infty y^{x-1}e^{-y}dy;
\eeqn
it can be viewed as a generalization of the factorial function in that it satisfies
\beqn\label{gammafact}
\Gamma(x+1)=x\Gamma(x).
\eeqn
While $\Gamma(n)=(n-1)!$ for all $n\in\Nat$, $n\not=0$, perhaps the most famous value of the gamma function on a non natural number is
\beqn\label{gammahalf}
\Gamma(\tfrac{1}{2})=\sqrt{\pi}.
\eeqn
{Of the many identities involving the gamma function, the one we use} is the Legendre duplication formula
\beqn\label{gammaldf}
\Gamma(x)\Gamma(x+\tfrac{1}{2})=2^{1-2x}\sqrt{\pi}\Gamma(2x).
\eeqn
and Euler's reflection formula 
\beqn\label{gammaref}
\Gamma(\tfrac{1}{2}-x)\Gamma(\tfrac{1}{2}+x)=\frac{\pi}{\cos(\pi x)}.
\eeqn

The beta function $B$ 
has many equivalent {definitions, of which the two useful} here are
\begin{align}
\label{betatrig}
B(x,y)&=2\int_0^{\pi/2} \sin^{2x-1}\theta\cos^{2y-1}\theta\, d\theta \\
\label{betatt}&=\int_0^1 t^{x-1}(1-t)^{y-1}dt,\qquad x,y>0.
\end{align}
{$B$ is related to $\Gamma$ by} the following relation:
\beqn\label{betagamma}
B(x,y)=\frac{\Gamma(x)\Gamma(y)}{\Gamma(x+y)}.
\eeqn

\medskip

\noindent
{\bf Acknowledgements.}
RP and PPG have conducted this research under the auspices of the Italian National Group for Mathematical Physics (GNFM) of the National Institute for Advanced Mathematics (INdAM). RP acknowledges support from the Project PRIN\,2017 no.\,20177TTP3S.

\bibliographystyle{acm}

\end{document}